\documentclass{sn-jnl}
\usepackage{graphicx} 

\usepackage[a4paper]{geometry}
\makeatletter
\let\c@author\relax
\makeatother

\usepackage[table]{xcolor}
\usepackage{multirow,multicol}
\usepackage{amsmath}

\usepackage{amsthm}
\usepackage{mathtools}
\usepackage{amssymb}
\usepackage{caption}
\usepackage{algorithm}
\usepackage{algpseudocode}
\usepackage{bm}

\usepackage{titlesec}
\usepackage{hyperref}
\usepackage{cleveref}

\newtheorem{theorem}{Theorem}[section]
\newtheorem{lemma}[theorem]{Lemma}

\newtheorem{corollary}[theorem]{Corollary}
\newtheorem{conjecture}[theorem]{Conjecture}

\theoremstyle{definition}
\newtheorem{definition}[theorem]{Definition}

\DeclareMathOperator{\LS}{RP}
\DeclareMathOperator{\ILS}{ILS}

\def \Z {\mathbb Z}

\usepackage{mhchem,multirow,calc,array,amsmath,booktabs}
\newlength\mylenA
\newlength\mylenB
\newlength\mylenC

\usepackage{color}

\usepackage{biblatex} 
\makeatletter
\def\ps@pprintTitle{%
  \let\@oddhead\@empty
  \let\@evenhead\@empty
  \def\@oddfoot{\reset@font\hfil\thepage\hfil}
  \let\@evenfoot\@oddfoot
}
\makeatother

\title{Latin squares with three disjoint subsquares of the same order}

\begin{document}

\author*[]{\fnm{Tara} \sur{Kemp}}\email{t.kemp@uq.net.au}

\author[]{\fnm{James G.} \sur{Lefevre}}\email{j.lefevre@uq.edu.au}

\affil[]{\orgdiv{School of Mathematics and Physics, ARC Centre of Excellence, Plant Success in Nature and Agriculture}, \orgname{The University of Queensland}, \orgaddress{\city{Brisbane}, \postcode{4072}, \state{Queensland}, \country{Australia}}}

\abstract{Given an integer partition $P = (h_1h_2\dots h_k)$ of $n$, a realization of $P$ is a latin square with disjoint subsquares of orders $h_1,h_2,\dots,h_k$. Most known results restrict either $k$ or the number of different integers in $P$. There is little known for partitions with arbitrary $k$ and subsquares of at least three orders. It has been conjectured that if $h_1=h_2=h_3\geq h_4\geq\dots\geq h_k$ then a realization of $P$ always exists. We prove this conjecture, and thus show the existence of realizations for many general partitions.}


\maketitle

\section{Preliminaries}

A \emph{latin square} of order $n$ is an $n\times n$ array $L$ filled with symbols from $[n]=\{1,2,\dots,n\}$ such that each symbol occurs exactly once in every row and column. A \emph{subsquare} is an $m\times m$ subarray of $L$ which is itself a latin square of order $m$ on some set of $m$ symbols. Subsquares are disjoint if they share no rows, columns or symbols.

The array in \Cref{fig: square example} is a latin square of order 9 with disjoint subsquares of orders 1, 2, and 3.

\begin{figure}[h]
    \centering
$$\arraycolsep=4pt\begin{array}{|c|c|c|c|c|c|c|c|c|}\hline
\cellcolor{lightgray}1 & \cellcolor{lightgray}2 & \cellcolor{lightgray}3 & 8 & 6 & 5 & 4 & 9 & 7 \\ \hline
\cellcolor{lightgray}2 & \cellcolor{lightgray}3 & \cellcolor{lightgray}1 & 7 & 9 & 4 & 5 & 6 & 8 \\ \hline
\cellcolor{lightgray}3 & \cellcolor{lightgray}1 & \cellcolor{lightgray}2 & 6 & 7 & 9 & 8 & 5 & 4 \\ \hline
7 & 6 & 9 & \cellcolor{lightgray}4 & \cellcolor{lightgray}5 & 8 & 1 & 3 & 2 \\ \hline
6 & 7 & 8 & \cellcolor{lightgray}5 & \cellcolor{lightgray}4 & 1 & 9 & 2 & 3 \\ \hline
9 & 8 & 5 & 3 & 2 & \cellcolor{lightgray}6 & \cellcolor{lightgray}7 & 4 & 1 \\ \hline
8 & 9 & 4 & 2 & 3 & \cellcolor{lightgray}7 & \cellcolor{lightgray}6 & 1 & 5 \\ \hline
5 & 4 & 7 & 9 & 1 & 3 & 2 & \cellcolor{lightgray}8 & 6 \\ \hline
4 & 5 & 6 & 1 & 8 & 2 & 3 & 7 & \cellcolor{lightgray}9 \\ \hline
\end{array}$$
    \caption{A latin square of order 9 with disjoint subsquares.}
    \label{fig: square example}
\end{figure}

Given an integer partition $P=(h_1\dots h_k)$ of $n$, a \emph{realization} of $P$, denoted $\LS(h_1\dots h_k)$, is a latin square of order $n$ with pairwise disjoint subsquares of orders $h_1,\dots,h_k$. The latin square in \Cref{fig: square example} is a realization of $(3^12^21^2)$. Realizations are also known as \emph{partitioned incomplete latin squares} (PILS).

Unless otherwise stated, we assume that $h_1\geq h_2\geq \dots\geq h_k > 0$. Also, the partition notation $(h_1^{m_1}h_2^{m_2}\dots h_\ell^{m\ell})$ represents a partition with $m_i$ parts of size $h_i$ for all $i\in[\ell]$. A realization is in normal form if the subsquares appear along the main diagonal, the $i^{th}$ subsquare is of order $h_i$, and for $i<j$ the symbols from the $i^{th}$ subsquare are less than the symbols from the $j^{th}$ subsquare.

L.~Fuchs first asked about the existence of realizations in terms of quasigroups with disjoint subquasigroups \cite{keedwell2015latin}. The problem of determining existence is partially solved, with most results concerning partitions with at most five parts or partitions with integers of at most two sizes.

Heinrich \cite{heinrich2006latin} determined existence for realizations with at most four disjoint subsquares, and Kemp \cite{kemp2024latin} extended this result to five disjoint subsquares.  D{\'e}nes and P{\'a}sztor \cite{denes1963some}, while studying quasigroups, completed the case with subsquares of only one size. Kuhl et al.~\cite{kuhl2018latin} completed the work started by Heinrich \cite{heinrich1982disjoint} on realizations with subsquares of at most two sizes.

We make use of the following results:

\begin{theorem}[\cite{heinrich2006latin}]
\label{thm: small k squares}
    Take a partition $(h_1h_2\dots h_k)$ of $n$ with $h_1\geq h_2\geq \dots\geq h_k > 0$. Then an $\LS(h_1h_2\dots h_k)$
    \begin{itemize}
        \item always exists when $k=1$;
        \item never exists when $k=2$;
        \item exists when $k=3$ if and only if $h_1=h_2=h_3$;
        \item exists when $k=4$ if and only if $h_1=h_2=h_3$, or $h_2=h_3=h_4$ with $h_1\leq 2h_4$.
    \end{itemize}
\end{theorem}

\begin{theorem}[\cite{denes1963some}]
\label{thm: a^k square}
    For $k\geq 1$ and $a\geq 1$, an $\LS(a^k)$ exists if and only if $k\neq 2$.
\end{theorem}

\begin{theorem}[\cite{heinrich1982disjoint,kuhl2018latin}]
\label{thm: squaresatmost2}
    For $a>b>0$ and $k>4$, an $\LS(a^ub^{k-u})$ exists if and only if $u\geq 3$, or $0< u < 3$ and $a\leq (k-2)b$.
\end{theorem}

Colbourn gave two conjectures about realizations in \cite{colbourn2018latin}. They concern two families of partitions for which realizations are believed to always exist. In this paper we prove the following conjecture.

\begin{conjecture}[Conjecture 1.8 of \cite{colbourn2018latin}]
\label{conj: largest 3}
    If $k\geq 3$, then an $\LS(h_1^3h_4\dots h_k)$ exists.
\end{conjecture}

In the remainder of this section, we give notation and definitions that will be used throughout the following sections.

Let $[m]=\{i \in \Z \mid 1\le i \le m\}$, $[m] + n = n+[m] = \{i \in \Z \mid n+1\le i \le n+m\}$, and for integer partition $P=(h_1h_2\dots h_k)$, define $P[i]$ to be the set $[h_i]+\sum_{j=1}^{j-1} h_j$, so that $\{P[i]\mid i \in [k]\}$ partitions $[\sum_{i=1}^k h_i]$.
Let $k\{x\}$ denote the multiset consisting of $k$ copies of element $x$, and so $\sum_{i=1}^n k_i\{x_i\} $ is the multiset consisting of $k_i$ copies of $x_i$ for $i\in [n]$.

\begin{definition}
    Given partitions $P,Q,R$ of $n$, where $P = (p_1\dots p_u)$, $Q = (q_1\dots q_v)$, $R = (r_1\dots r_t)$, let $O$ be a $u\times v$ array of multisets, with elements from $[t]$. For $i\in [u]$ and $j\in[v]$, let $O(i,j)$ be the multiset of symbols in cell $(i,j)$ and let $|O(i,j)|$ be the number of symbols in the cell, including repetition.

    Then $O$ is an \emph{outline rectangle} associated to $(P,Q,R)$ if
    \begin{enumerate}
        \item $|O(i,j)| = p_iq_j$, for all $(i,j)\in[u]\times[v]$;
        \item symbol $l\in[t]$ occurs $p_ir_l$ times in the row $(i,[v])$;
        \item symbol $l\in[t]$ occurs $q_jr_l$ times in the column $([u],j)$.
    \end{enumerate}
\end{definition}

The array of multisets in \Cref{fig: outline rectangle example} is an outline rectangle associated to $((1^32^21^2),(3^12^21^2),(3^11^6))$.

\begin{figure}[h]
    \centering
$$\arraycolsep=4pt\begin{array}{|ccc|cc|cc|c|c|}\hline
1 & 1 & 1 & 4 & 6 & 2 & 3 & 7 & 5 \\ \hline
1 & 1 & 1 & 5 & 7 & 2 & 3 & 4 & 6 \\ \hline
1 & 1 & 1 & 4 & 5 & 6 & 7 & 3 & 2 \\ \hline
4 & 4 & 5 & 2 & 2 & 1 & 1 & 1 & 1 \\
5 & 6 & 7 & 3 & 3 & 6 & 7 & 1 & 1 \\ \hline
2 & 3 & 6 & 1 & 1 & 4 & 4 & 1 & 1 \\
6 & 7 & 7 & 1 & 1 & 5 & 5 & 2 & 3 \\ \hline
2 & 3 & 5 & 1 & 7 & 1 & 1 & 6 & 4 \\ \hline
2 & 3 & 4 & 1 & 6 & 1 & 1 & 5 & 7 \\ \hline
\end{array}$$
    \caption{An outline rectangle associated to $((1^32^21^2),(3^12^21^2),(3^11^6))$.}
    \label{fig: outline rectangle example}
\end{figure}

Outline rectangles were introduced by Hilton in \cite{hilton1980reconstruction} and can be obtained from latin squares.

\begin{definition}
    Given partitions $P,Q,R$ of $n$, where $P = (p_1\dots p_u)$, $Q = (q_1\dots q_v)$ and $R = (r_1\dots r_w)$, and a latin square $L$ of order $n$, the \emph{reduction modulo $(P,Q,R)$} of $L$, denoted $O$, is the $u\times v$ array of multisets obtained by amalgamating rows $(p_1 + \dots + p_{i-1}) + [p_i]$ for all $i\in[u]$, columns $(q_1 + \dots + q_{j-1}) + [q_j]$ for all $j\in[v]$, and symbols $(r_1 + \dots + r_{k-1}) + [r_k]$ for all $k\in[w]$.

    When amalgamating symbols, for $k\in[w]$ we map all symbols in $(r_1 + \dots + r_{k-1}) + [r_k]$ to symbol $k$.
\end{definition}

The outline rectangle in \Cref{fig: outline rectangle example} is a reduction modulo $((1^32^21^2),(3^12^21^2),(3^11^6))$ of the latin square given in \Cref{fig: square example}.

It is clear that a reduction modulo $(P,Q,R)$ of a latin square is an outline rectangle associated to $(P,Q,R)$. If the reverse is true, meaning that an outline rectangle $O$ is a reduction modulo $(P,Q,R)$ of a latin square $L$, then we say that $O$ \emph{lifts} to $L$. As shown by Hilton \cite{hilton1980reconstruction}, the reverse is always true.

\begin{theorem}[\cite{hilton1980reconstruction}]
\label{outline rectangle to square}
    Let $P,Q,R$ be partitions of $n$. For every outline rectangle $O$ associated to $(P,Q,R)$, there is a latin square $L$ of order $n$ such that $O$ lifts to $L$.
\end{theorem}

When we want the latin square $L$ to have disjoint subsquares, the outline rectangle $O$ must have multisets that correspond to the subsquares.

\begin{lemma}[\cite{kuhl2019existence}]
\label{outline to realization}
    For a partition $P = (h_1\dots h_k)$, an outline rectangle associated to $(P,P,P)$ with cell $(i,i)$ filled with $h_i^2$ copies of symbol $i$ for all $i\in[k]$ lifts to a realization of $P$.
\end{lemma}

When $P=Q=R$, an outline rectangle associated to $(P,P,P)$ is called an \emph{outline square} associated to $P$.

In Section 2 we use outline rectangles to construct realizations of the form given in \Cref{conj: largest 3} that satisfy an additional condition. Then in Section 3, frequency arrays are used to construct all remaining cases.

\section{Circulant construction}

The primary construction in this section is based on the prolongation of an odd order back circulant latin square, followed by amalgamation of selected symbols, then intercalate trades to provide the required subsquares. We do not formally define these terms as the following lemma provides an explicit construction.

\begin{lemma}
\label{circulant_construction}
    Let $P=(h_1h_2h_3\dots h_k)$, where $n=\sum_{i=1}^k h_i$, $t=n-h_1$ is odd, and $h_i\ge h_{i+1}$ for $2\le i \le k-1$. Further suppose that $h_2 \le (t+1)/4$ and $2h_2 \le h_1 \le t+1-2h_2$.
    Define a multiset $V = \sum_{i=2}^k h_i\{h_1+i-1\}$, so $|V|=t$.
    Then there exists an outline rectangle $O$ associated to $((1^n), (1^n), (1^{h_1}h_2h_3\dots h_k))$ with the following properties:
    \begin{enumerate}
        \item For any $x,y \in P[1]$, then $O(x,y)\in P[1]$. For any
        $i \in [k]\setminus \{1\}$ and $x,y \in P[i]$, then $O(x,y)=\{h_1+i-1\}$. Thus $O$ lifts to a realization of $P$ in normal form.
        \item We have triple sets $T_i=\{(x_{i,j},y_{i,j},z_{i,j}) \mid j\in [t]\}$, for all $i\in [h_1-2h_2]$, such that: $\{x_{i,j} \mid j\in [t]\}=\{y_{i,j} \mid j\in [t]\} =  [n]\setminus [h_1]$; $\{z_{i,j} \mid j\in [t]\}=V$; the pairs $(x_{i,j},y_{i,j})$ are distinct for all $i\in [h_1-2h_2]$ and $j \in [t]$; and for any $i\in [h_1-2h_2]$ and $j \in [t]$, $O(x_{i,j},y_{i,j})=\{i\}$ and $O(x_{i,j},i)=O(i,y_{i,j})=\{z_{i,j}\}$.
    \end{enumerate}
\end{lemma}
\begin{proof}
    Define $\oplus_t$ such that $x_1 \oplus_t x_2 \equiv x_1+x_2 \pmod{r}$ and $x_1 \oplus_t x_2 \in [r]$, and define $\otimes_t$ and $\ominus_t$ similarly.
    Observe that $([t],\oplus_t,\otimes_t)$ is a commutative ring in which 2 has multiplicative inverse $\frac{t+1}{2}$.
    We can form a latin square of order $t$ by placing $(i \oplus_t j) \otimes_t \frac{t+1}{2}$ in cell $(i,j)$ for all $i,j \in [t]$. For any $d\in [t]$, the set of cells $(i,j)$ with $i \ominus_t j = d$ forms a {\em transversal}, in which each symbol from $[t]$ occurs exactly once. Let 
    \begin{eqnarray*}
        D_1 &=& \{t-(2h_2-2),t-(2h_2-4),\dots,t-4,t-2\} \cup \{t\}\cup \{2,4,\dots ,2h_2-4,2h_2-2\}, \\
        D_2 &=& \{t-(2h_2-1),t-(2h_2-3),\dots ,t-3,t-1\} \cup \{1,3,\dots,2h_2-3,2h_2-1\}.
    \end{eqnarray*}
    These are subsets of $[t]$ of size $2h_2-1$ and $2h_2$ respectively, which are disjoint since $h_2 \le (t+1)/4$. Let $D=\{d_i \mid i\in [h_1]\}$ be a subset of [t] of size $h_1$ with $\{d_i \mid i \in [h_1]\setminus [h_1-2h_2]\} = D_2$ and $\{d_i \mid i \in [h_1-2h_2]\} \subseteq [t] \setminus (D_1\cup D_2)$.

    We construct a latin square $L$ of order $n$ as follows. Consider $i,j \in [t]$. If $j \ominus_t i = d_k \in D$ for some $k \in [h_1]$, then we place $k$ in cell $(i+h_1,j+h_1)$, and $((i \oplus_t j) \otimes_t \frac{t+1}{2})+h_1$ in cells $(k,j+h_1)$ and $(i+h_1,k)$. Otherwise, we place $((i \oplus_t j) \otimes_t \frac{t+1}{2})+h_1$ in cell $(i+h_1,j+h_1)$. We complete $L$ by filling the cells with row and column in $[h_1]$ with a subsquare on symbols $[h_1]$. 

    We now form an outline rectangle $O_1$ associated to $(1^n,1^n,1^{h_1-2h_2}(2h_2)^1h_2h_3\dots h_k)$ by amalgamating the symbols in $[h_1]\setminus[h_1-2h_2]$ to 0 and $P[h_i]$ to $h_1+i-1$ for each $i \in [k] \setminus \{1\}$. Note that the symbols are now non-contiguous integers. 
    Observe that for $a,b \in [t]$, if $b \ominus_t a \in D_2$, then $O_1(h_1+a,h_1+b)=\{0\}$.

For each $i \in [k]\setminus\{1\}$, we define 
$$R_i = \{ (a,b) \mid a,b \in [h_i], \: a < b, \: b-a \text{ odd} \}, \;\;\; S^i=\sum_{j=2}^{i-1} h_j. $$ 
Consider any $i\in [k]\setminus\{1\}$ and $(a,b)\in R_i$, and let 
\begin{eqnarray*}
    x_1 &=& S^i+a, \\
    y_1 &=& S^i+b, \\
    x_2 &=& (S^i+1) \ominus_t a, \\
    y_2 &=& (S^i+2h_i+1) \ominus_t b.
\end{eqnarray*}
 We have
\begin{eqnarray}
    (x_1 \oplus_t y_1) \otimes_t \left(\frac{t+1}{2}\right) &=& \left(S^i +\frac{a+b-1}{2}\right) \oplus_t \left(\frac{t+1}{2}\right), \label{x1py1} \\
    (x_2 \oplus_t y_2) \otimes_t \left(\frac{t+1}{2}\right) &=& \left(S^i +h_i-\frac{a+b-1}{2}\right)\oplus_t\left(\frac{t+1}{2}\right), \label{x2py2}\\
    (x_2 \oplus_t y_1) \otimes_t \left(\frac{t+1}{2}\right) &=& S^i +1+\frac{b-a-1}{2}, \label{x2py1} \\
     (x_1 \oplus_t y_2) \otimes_t \left(\frac{t+1}{2}\right) &=& S^i +h_i-\frac{b-a-1}{2}, \label{x1py2}
\end{eqnarray}
and
\begin{eqnarray}
y_1 \ominus_t x_1 &=& b-a \in D_2 \cup [1,h_i-1], \label{y1mx1}\\
y_2 \ominus_t x_2 &=& 2h_i - (b-a) \in D_2 \cup[h_i+1,2h_i-1], \label{y2mx2}\\
y_1 \ominus_t x_2 &=& a+b-1 \in D_1 \cup [2,2h_i-2], \label{y1mx2}\\
y_2 \ominus_t x_1 &=& 2h_i+1-(a+b) \in D_1 \cup [2,2h_i-2].\label{y2mx1}
\end{eqnarray}
Consider the pairs $(x_1,y_1),(x_2,y_2),(x_1,y_2),(x_2,y_1)$. We first show that these are distinct from each other and from the corresponding pairs for all other $i,a,b$. Eqs (\ref{y1mx1}) and (\ref{y2mx2}) distinguish $(x_1,y_1)$ and $(x_2,y_2)$ pairs for the same $i$, while (\ref{x1py1}) and (\ref{x2py2}) distinguish them from the $(x_1,y_1)$ and $(x_2,y_2)$ pairs for other $i$ values. Similarly, Eqs (\ref{x2py1}) and (\ref{x1py2}) distinguish $(x_1,y_2)$ and $(x_2,y_1)$ from each other and from the corresponding pairs for other $i$. Eqs
(\ref{y1mx1}-\ref{y2mx1}) distinguish all $(x_1,y_1),(x_2,y_2)$ pairs from all $(x_2,y_1),(x_1,y_2)$ pairs, and also all such pairs from the corresponding pairs with reversed order (that is, exchanging the role of rows and columns).

By Eqs (\ref{y1mx1}-\ref{y2mx1}), we have $O_1(y_1+h_1,x_1+h_1)=O_1(y_2+h_1,x_2+h_1)=\{0\}$, while cells $(y_2+h_1,x_1+h_1)$ and $(y_1+h_1,x_2+h_1)$ of $L$ contain $(x_1 \oplus_t y_2) \otimes_t \left(\frac{t+1}{2}\right)+h_1$ and $(x_2 \oplus_t y_1) \otimes_t \left(\frac{t+1}{2}\right)+h_1$ respectively. These both lie in the set $P[i]$, so $O_1(y_1+h_1,x_2+h_1)=O_1(y_2+h_1,x_1+h_1)=\{h_1+i-1\}$. 
We form the outline rectangle $O_2$ from $O_1$ by swapping the values 0 and $h_1+i-1$ in these four cells, for each $i \in [k]\setminus\{1\}$ and $(a,b)\in R_i$, and similarly for the equivalent cells with row and column reversed. Consider any $i \in [k]\setminus\{1\}$ and  $a,b \in P[i]$. Then $a \ominus_t b$ is in $D_1$ or $D_2$. If $a \ominus_t b \in D_2$, the trades we have described give $O_2(a,b) = \{h_1+i-1\}$, as required (while $O_1(a,b) = \{0\}$). If $a \ominus_t b \in D_1$, then we can assume without loss of generality that $b=a+2d$, where $0\le d\le (h_i-1)$, so $(a \oplus_t b) \otimes_t \left(\frac{t+1}{2}\right)+h_1 = (a \oplus_t d) + h_1 \in P[i]$. Again, $O_2(a,b)=\{h_1+i-1\}$.

We construct the final outline rectangle $O$ associated to $(1^n,1^n,1^{h_1}h_2h_3\dots h_k)$ by lifting $O_2$ to a latin square, with 0 mapped to $[h_1]\setminus [h_1-2h_2]$, then for each $i\in [k] \setminus \{1\}$ re-amalgamating the $h_i$ symbols corresponding to $h_1+i-1$ in $O_2$ back to $h_1+i-1$.

It remains only to establish the existence of the triple sets $T_i$. These correspond to the differences in $D \setminus D_2$.
For each $i\in [h_1-2h_2]$, $T_i = \{(a+h_1,b+h_1,f(a,b)) \mid a,b \in [t],\; b\ominus_t a = d_i\}$, where $f(a,b)$ is the symbol $(a \oplus_t b) \otimes_t \left(\frac{t+1}{2}\right)+h_1$ after amalgamation of each set $P[j]$ to $h_1+j-1$, $j\in [k]\setminus\{1\}$. By construction, cell $(a+h_1,b+h_1)$ of $L$ contains $i$, while cells $(i,b+h_1)$ and $(a+h_1,i)$ contain $(a \oplus_t b) \otimes_t \left(\frac{t+1}{2}\right)+h_1$. This last value is amalgamated in $O_1$, but the symbols in $[h_1-2h_2]$ are not. Since $d_i \not\in D_1 \cup D_2$, none of the three cells is involved in any trade, so the structure is retained in $O_2$ and hence $O$.

\end{proof}

\begin{corollary}
\label{odd_r_from_circulant}
    Let $P=(h_1h_2h_3h_4\dots h_k)$, where $h_1=h_2=h_3$, $h_4 \ge \dots \ge h_k$, $r=\sum_{i=4}^k h_i$ is odd, $h_4 \le (r+1)/4$, and $2h_4 \le 3h_1 \le r+1-2h_4$.
    Then there exists an outline rectangle $O$ associated to $(P,P,P)$ such that $O(i,i)$ contains $h_i^2$ copies of $i$ for each $i\in [k]$, and additionally $O(1,2)$ and $O(2,1)$ each contain $h_1^2$ copies of 3, $O(1,3)$ and $O(3,1)$ each contain $h_1^2$ copies of 2, and $O(2,3)$ and $O(3,2)$ each contain $h_1^2$ copies of 1.  
\end{corollary}
\begin{proof}
    Let $n=\sum_{i=1}^k h_i = r+3h_1$. By \Cref{circulant_construction} (with $t=r$), we have an outline rectangle associated to $((1^n), (1^n), (1^{3h_1}h_4h_5\dots h_k))$ which lifts to a realization of $((3h_1)h_4\dots h_k)$ in normal form. We construct an outline rectangle $O$ by amalgamating rows, columns and symbols according to partition $P$, so that $O(i,i)$ contains $h_i^2$ copies of $i$ for each $i\in [k] \setminus [3]$. The cells $\{O(i,j) \mid i,j \in [3]\}$ form a subsquare on the amalgamated symbols $[3]$, so we can replace these cells with the required multisets.
\end{proof}

The even $r$ case is more complex.

\begin{lemma}
\label{even_r_from_circulant}
    Let $P=(h_1h_2h_3h_4\dots h_k)$, where $h_1=h_2=h_3\ge 2$, $h_4 \ge \dots \ge h_k$, $h_1(h_1-1)\geq 2(h_k-1)$, $r=\sum_{i=4}^k h_i$ is even, $h_4 \le (r-2)/4$, and 
    $2h_4+2 \le 3h_1+1 \le r+1-2h_4$.
    
    Then there exists an outline rectangle $O$ associated to $(P,P,P)$ such that $O(i,i)$ contains $h_i^2$ copies of $i$ for each $i\in [k]$, and additionally $O(1,2)$, $O(2,3)$ and $O(3,1)$ contain $h_1^2$ copies of 3, 1, and 2 respectively, while 
    $O(1,3)$, $O(2,1)$ and $O(3,2)$ contain at least $h_1(h_1-1)-2(h_k-1)$ copies of 2, 3, and 1 respectively.

\end{lemma}
\begin{proof}
    Let $n=r+3h_1$.
By \Cref{circulant_construction} (with $t=r-1$), we have an outline rectangle $O_1$ associated to $((1^n), (1^n), (1^{3h_1+1}h_4h_5\dots (h_k-1)))$ which lifts to a realization of $((3h_1+1)h_4\dots (h_k-1))$ in normal form. The rows and columns $[3h_1+1]$ form a subsquare on symbols $[3h_1+1]$, and for $i\in [k-1] \setminus [3]$ the rows and columns $P[i]+1$ are filled with the symbol $3h_1+i-2$, while the rows and columns $(P[k]+1)\setminus \{n+1\}$ are filled with the symbol $3h_1+k-2$.
Further, since $(3h_1+1)-2h_4 \ge 2$, then there exist triple sets $T_1, T_2$ of size $r-1$ such that for $i\in [2]$ the following hold: (1) for any $x\in [n]\setminus [3h_1+1]$, there is exactly one triple $(x,y,z) \in T_i$ and one triple $(x,y,z) \in T_i$; (2) $\{z \mid (x,y,z) \in T_i$ is a multiset equal to $V= \left(\sum_{j=4}^{k} h_j \{3h_1+j-2\}\right)-\{3h_1+k-2\}$; and for any $(x,y,z) \in T_i$, cell $(x,y)$ contains $i$ and cells $(y,i)$ and $(i,x)$ contain $z$. 

    Let $P' = (h_1^31^r)$. We form an outline rectangle $O_2$ associated to $(P',P',(h_1^3 h_4h_5\dots h_{k-1}(h_k-1)1)$ by applying the following amalgamation and relabelling to $O_1$: rows, columns and symbols $[3h_1]$ are amalgamated into 3 groups of size $h_1$, labelled 1,2,3, such that the original values 1 and 2 are joined to groups 1 and 2 respectively. The row and column values $[n] \setminus [3h_1+1]$ are each reduced by $3h_1-2$, mapping to the range $4$ to $r+2$, while $3h_1+1$ becomes $r+3$. Similarly, the symbol values $[k+3h_1-2] \setminus [3h_1+1]$ are each reduced by $3h_1-2$, mapping to the range $4$ to $k$, while $3h_1+1$ becomes $k+1$.
    Finally, rows 2 and 3 are swapped, columns 1 and 3 are swapped, and symbols 1 and 2 are swapped. 
    The cells $O_2(i,j)$, $i,j \in \{1,2,3,r+3\}$, contain only elements of $\{1,2,3,k+1\}$, and thus represent a subsquare. So we can make the substitution 
\newline

    $\begin{array}{c|c|c|c|c|}
O_2(i,j) & j=1 & j=2 & j=3 & j=r+3 \\ 
\hline
i=1 & h_1^2 \{1\} & h_1^2\{3\} & h_1(h_1-1)\{2\} & h_1\{2\} \\
&&&+h_1\{k+1\}& \\
\hline
i=2 & h_1(h_1-1)\{3\} & h_1^2\{2\} & h_1^2 \{1\} & h_1\{3\} \\
&+h_1\{k+1\}&&& \\
\hline
i=3 & h_1^2\{2\} & h_1(h_1-1)\{1\} & h_1^2\{3\} & h_1\{1\} \\
&&+h_1\{k+1\}&& \\
\hline
i=r+3 & h_1\{3\} & h_1\{1\} & h_1\{2\} & \{k+1\} \\
\hline
\end{array}$
\newline

    We now form an outline rectangle $O_3$ associated to $(P',P',P)$ by further amalgamating the symbols $k$ and $k+1$ in $O_2$, with the new symbol labelled $k$.
    Observe that for each $i \in [k]\setminus \{3\}$, and every $x,y\in (P[h_i]-3h_1+3)$, $O_3(x,y)=\{i\}$, with the exception of $i=k$ and exactly one of $x,y$ equal to $r+3$.
    To complete the construction of $O$, we therefore identify trades that replace the multisets $O_3(x,r+3)$ and $O_3(r+3,x)$ with $\{k\}$, for all $x \in (P[h_k]-3h_1+3) \setminus \{r+3\}$. 

    For $i=1,2$, let $T'_i = \{(y-3h_1+2,x-3h_1+2,z-3h_1+2) | (x,y,z) \in T_i\}$.
    Then for each $(x,y,z) \in T'_1$, we have $O_3(x,y)=\{2\}$ and $z \in O_3(1,y), O_3(x,3)$, and for each $(x,y,z) \in T'_2$, we have $O_3(x,y)=\{1\}$ and $z \in O_3(x,2), O_3(3,y)$.

    We identify four sequences of size $h_k-1$: for $i \in [h_k-1]$ let $a_{i} = r+3-i$, and let $b_i < r+3$, $c_i<r+4-h_k$ and $d_i$ be the values satisfying $O_3(r+3,b_i)=\{k\}$, $O_3(c_{i},a_{i})=\{k\}$, and $O_3(r+3,a_{i})=\{d_{i}\}$. Note that $b_i,c_i,d_i \in [r+2]\setminus [3]$. For $b_i$ and $d_i$, this follows from the fact that for each $x\in[3]$, $O_3(r+3,x)=O_2(r+3,x)=h_1{y}$ for some $y\in[3]$. For $c_i$, we observe that $O_3(x,a_i)=O_2(x,a_i)=\{k\}$ for all $x \in [r+2]\setminus [r+3-h_k]$. The cell $(c_i,a_i)$ contains the remaining copy of $k$ in column $a_i$ of $O_3$, which corresponds to the copy of $k+1$ in column $a_i$ of $O_2$. In the array above we see that rows $1,2,3,r+3$ of $O_2$ contain all available copies of $k+1$ in columns $1,2,3,r+3$, so $c_i\in [r+2]\setminus [3]$.
    Let $U = \{(y_j,x_j,z_j) \mid j \in [u]\}$ be the minimal subset of $T'_1$ such that $a_i, b_i \in \{x_j \mid j \in [u]\}$, $c_i \in \{y_j \mid j \in [u]\}$, $d_i \in \{z_j \mid j \in [u]\}$. Since $\{a_i \mid i \in [h_k-1]\}$ and $\{b_i \mid i \in [h_k-1]\}$ must be disjoint, $2(h_k-1) \le u=|U| \le 4(h_k-1)$. The trade is defined in the following table:
    \newline

    $\begin{array}{c|c|c|c}
    \text{cell} & \text{parameter} & \text{removed} & \text{added} \\
    \hline
    (y_j,x_j) & j\in [u] & \{2\} & \{z_j\} \\
    (y_j,3) & j\in [u] & \{z_j\} & \{2\} \\
    (1,x_j) & j\in [u] & \{z_j\} & \{2\} \\
    (r+3,a_i) & i\in [h_k-1] & \{d_i\} & \{k\} \\
    (r+3,b_i) & i\in [h_k-1] & \{k\} & \{2\} \\
    (r+3,3) && (h_k-1)\{2\} & \{d_i \mid i \in [h_k-1]\} \\
    (c_i,3) & i\in [h_k-1] & \{2\} & \{k\} \\
    (c_i,a_i) & i\in [h_k-1] & \{k\} & \{2\} \\
    (1,a_i) & i\in [h_k-1] & \{2\} & \{d_i\} \\
    (1,b_i) & i\in [h_k-1] & \{2\} & \{k\} \\
    (1,3) & & u\{2\} & \{z_j \mid j\in [u] \} \\
    (1,3) & & \{d_i \mid i\in [h_k-1] \} & (h_k-1)\{2\}\\
    (1,3) & & (h_k-1)\{k\} & (h_k-1)\{2\} \\
\end{array}$
\newline

This trade replaces $O_3(r+3,i)$ with $\{k\}$ for $i\in (P[h_k]-3h_1+3) \setminus \{r+3\}$. A similar trade replaces $O_3(i,r+3)$ with $\{k\}$: rows and columns are reversed, and $T_2$, row 3, column 2 and symbol 1 are used in place of $T_1$, row 1, column 3, and symbol 2. These trades are disjoint from each other and from the remainder of the required subsquares. We apply both of these trades and then amalgamate rows and columns to give $O$ associated to $(P,P,P)$; each row and column set $P[i]-3h_1+3$ is mapped to $i$, for $i\ge 4$. $O_3(1,3)$ contains $h_1(h_1-1)$ copies of $2$, which is reduced by $u-2(h_k-1)$ in the trade. Thus $O(1,3)$ contains $h_1(h_1-1)+2(h_k-1)-u \ge h_1(h_1-1)-2(h_k-1)$ copies of $2$, and similarly $O(3,2)$ contains at least $h_1(h_1-1)-2(h_k-1)$ copies of $1$.

\end{proof}

\begin{lemma}
\label{blow_up_3}
Let $P=(h_1h_2h_3h_4\dots h_k)$, where $h_1=h_2=h_3$, $h_4 \ge \dots \ge h_k$, and $r=\sum_{i=4}^k h_i$. Suppose that there is an outline rectangle $O$ associated to $(P,P,P)$, such that $O(i,i)$ contains $h_i^2$ copies of $i$ for each $i\in [k]$. Let integers $\beta_1, \beta_2 \ge 0$ be such that $O(2,3)$, $O(3,1)$, and $O(1,2)$ contain at least $\beta_1$ copies of 1, 2, and 3  respectively, and $O(3,2)$, $O(1,3)$, and $O(2,1)$ contain at least $\beta_2$ copies of 1, 2, and 3 respectively.
Let $P'=(g^3h_4\dots h_k)$, such that $g\ge h_1$ and $r(g-h_1) \le 2(g^2-h_1^2)+\beta_1+\beta_2$. 
Then there exists an outline rectangle $O'$ associated to $(P',P',P')$, such that $O'(i,i)$ contains $h_i^2$ copies of $i$ for each $i\in [k] \setminus [3]$, and $g^2$ copies of $i$ for each $i\in [3]$.
\end{lemma}
\begin{proof}
    Choose integers $p,q$ such that $0\le p \le g^2-h_1^2+\beta_1$, $0\le q \le g^2-h_1^2+\beta_2$, and $p+q = r(g-h_1)$. Choose further non-negative integers $p_i, q_i$, for $i\in [k]\setminus [3]$, such that $p_i+q_i=h_i(g-h_1)$, $\sum_{i=4}^k p_i =p$, and $\sum_{i=4}^k q_i =q$. Define multisets $S_1=\sum_{i=4}^k p_i\{i\}$, $S_2=\sum_{i=4}^k q_i\{i\}$. For any $i,j \in [k]\setminus [3]$, let 
    \begin{eqnarray*}
    O'(1,1) &=& g^2 \{1\}, \\
    O'(2,2) &=& g^2 \{2\}, \\
    O'(3,3) &=& g^2 \{3\}, \\
    O'(1,2) &=& O(1,2) + S_1 - (p-g^2+h_1^2)\{3\}, \\
    O'(2,3) &=& O(2,3) + S_1 - (p-g^2+h_1^2)\{1\}, \\
    O'(3,1) &=& O(3,1) + S_1 - (p-g^2+h_1^2)\{2\}, \\
    O'(2,1) &=& O(2,1) + S_2 - (q-g^2+h_1^2)\{3\}, \\
    O'(3,2) &=& O(3,2) + S_2 - (q-g^2+h_1^2)\{1\}, \\
    O'(1,3) &=& O(1,3) + S_2 - (q-g^2+h_1^2)\{2\}, \\
    O'(1,j) &=& O(1,j) + p_j\{3\} + q_j\{2\}, \\
    O'(2,j) &=& O(2,j) + p_j\{1\} + q_j\{3\}, \\
    O'(3,j) &=& O(3,j) + p_j\{2\} + q_j\{1\}, \\
    O'(i,1) &=& O(i,1) + p_i\{2\} + q_i\{3\}, \\
    O'(i,2) &=& O(i,2) + p_i\{3\} + q_i\{1\}, \\
    O'(i,3) &=& O(i,3) + p_i\{1\} + q_i\{2\}, \\
    O'(i,j) &=& O(i,j)   
\end{eqnarray*}
\end{proof}

\begin{theorem}
Let $3\le m < k$. There exists a realization of $(h_1^mh_{m+1}\dots h_k)$ if $h_1 = \dots = h_m \ge h_{m+1} \ge \dots \ge h_k$ and $(m-1)(h_1+h_{m+1}) < \sum_{i=m+1}^k h_i$. 
\label{jl_combined_construction}
\end{theorem}
\begin{proof}

Let $r=\sum_{i=4}^k h_i$, so the condition can be written as $(2m-4)h_1+(m-1)h_{m+1}<r$. If $h_1>h_4$ then $m=3$ and $2(h_1+h_4)<r$ and hence $4h_4 < r-2$. If $h_1=h_4$ then $m \ge 4$, so $4h_1+3h_{m+1}<r$, giving $4h_4=4h_1<r-3$. So in all cases $4h_4<r-2$.
We may assume $h_1 \ge 3$, $h_4 \ge 2$, and $h_4 \ge h_k+1$, since otherwise there are at most two distinct subsquare sizes, and the result follows from \Cref{thm: small k squares,thm: a^k square,thm: squaresatmost2}.

First suppose that $r$ is odd. By \Cref{odd_r_from_circulant} and \Cref{blow_up_3}, there exists an outline rectangle that lifts to the required realization if there exists an integer $t$ such that $2h_4 \le 3t \le r+1-2h_4$ and $0 \le r(h_1-t) \le 2h_1^2$; that is, if  $\max\{\frac{2}{3}h_4,h_1-\frac{2}{r}h_1^2 \}\leq t\leq \min\{ \frac{1}{3}(r+1-2h_4), h_1 \}$. 
If $h_1 \le \frac{1}{3}(r+1-2h_4)$ then we choose $t=h_1$.
Otherwise, $t$ exists provided that 
$\frac{1}{3}(r+1-2h_4) - \max\{\frac{2}{3}h_4,h_1-\frac{2}{r}h_1^2 \} \ge \frac{2}{3}$.
Now $4h_4 \leq r-4$, so $\frac{1}{3}(r+1-2h_4) - \frac{2}{3}h_4 \geq \frac{5}{3}$. Thus we are left to prove that
$\frac{1}{3}(r+1-2h_4) - h_1+\frac{2}{r}h_1^2 \geq \frac{2}{3}$; that is,
$6h_1^2-3rh_1+r(r-1-2h_4) \geq 0$. This holds for all $h_1$ provided that the quadratic discriminant is less than or equal to zero. This reduces to the condition $8+16h_4 \leq 5r$, which follows from $h_4 \leq \frac{r-3}{4}$.

Now suppose that $r$ is even. Then $4h_4 < r-2$ implies that $h_4 \le (r-4)/4$. By \Cref{even_r_from_circulant,blow_up_3}, there exists an outline rectangle that lifts to the required realization if $t(t-1)\ge 2(h_k-1)$, $2h_4+2 \le 3t+1 \le r+1-2h_4$ and $0 \le r(h_1-t) \le 2(h_1^2-t^2) + t^2 + t(t-1)-2(h_k-1) = 2h_1^2 - t - 2(h_k-1)$ for some integer $t$. 
If the second condition holds then $t\geq \frac{2}{3}h_4+\frac{1}{3}$. Together with the assumption that $h_4\geq h_k+1$ this gives 
$t(t-1) - 2(h_k-1) \geq \frac{4}{9}h_k^2-\frac{4}{3}h_k+2 = \frac{4}{9}(h_k-\frac{3}{2})^2+1\geq 1$, satisfying the first condition.
Therefore the realisation exists if there is an integer $t$ satisfying 
$$\max\left\{\frac{1}{3}(2h_4+1),\frac{1}{r-1}(rh_1-2h_1^2+2h_k-2) \right\}\leq t\leq \min\left\{ \frac{1}{3}(r-2h_4), h_1 \right\}.$$
If $h_1 \le \frac{1}{3}(r-2h_4)$ then we choose $t=h_1$; the lower bound follows immediately from $h_1 \ge h_4 \ge h_k+1$ and $h_1 \ge 2$. 
Otherwise, $t$ exists provided that 
$\frac{1}{3}(r-2h_4) - \max\{\frac{1}{3}(2h_4+1),\frac{1}{r+1}(rh_1-2h_1^2-2h_k+2) \} \ge \frac{2}{3}$.
Now $4h_4 \leq r-4$, so $\frac{1}{3}(r-2h_4) - \frac{1}{3}(2h_4+1) \geq 1$. Thus we are left to prove that
$\frac{1}{3}(r-2h_4) - \frac{1}{r-1}(rh_1-2h_1^2+2h_k-2) \geq \frac{2}{3}$; that is,
$6h_1^2-3rh_1+(r-1)(r-2h_4-2) -6(h_k-1) \geq 0.$
Since $h_k \le h_4-1$, it is sufficient to prove that
$$6h_1^2-3rh_1+r^2-3r-2h_4r-4h_4+14 \geq 0.$$

This holds for all $h_1$ provided that the quadratic discriminant is less than or equal to zero. This reduces to the condition
$5r^2-r(24+16h_4)+112-32h_4 \geq 0$.
Since $4h_4 \le r-4$, $5r^2-r(24+16h_4)+112-32h_4 \ge r^2 -16r +144 \ge 80$, so we are done.
\end{proof}

\section{Frequency arrays}

In this section, we build on the realizations found in the previous section to prove our main result.

Frequency arrays were introduced in \cite{kemp2025further} and are similar to outline squares.

\begin{definition}
    A \emph{frequency array} $F$ of order $k$ is a $k\times k$ array, where each cell contains a single non-negative integer.
\end{definition}

\begin{definition}
    Let $O$ be a $k\times k$ array of multisets. $O(i,j)$ denotes the multiset of symbols in cell $(i,j)$, $O_{\ell}^i$ and $^jO_\ell$ denote the number of copies of symbol $\ell$ in row $i$ and column $j$ respectively.
    Then $O$ is an \emph{outline array} corresponding to a frequency array $F$ of order $k$, if
    \begin{itemize}
        \item $|O(i,j)| = F(i,j)$,
        \item $O_\ell^i = F(i,\ell)$, and
        \item $^jO_\ell = F(\ell,j)$.
    \end{itemize}
\end{definition}

An outline square associated to $P = (h_1\dots h_k)$ is equivalent to an outline array for a frequency array of order $k$ with $F(i,j) = h_ih_j$.

The proofs of \Cref{lemma: sum freq arrays,lemma: summing rows/columns of outline array} can be found in \cite{kemp2025further}.

\begin{lemma}
\label{lemma: sum freq arrays}
    If $O_1$ and $O_2$ are outline arrays corresponding to the frequency arrays $F_1$ and $F_2$ respectively, then there exists an outline array $O^*$ corresponding to the frequency array $F^*$ where $F^*(i,j) = F_1(i,j) + F_2(i,j)$.
\end{lemma}

Observe that a realization $\LS(h_1\dots h_k)$ is equivalent to the sum of outline arrays $O_1$ and $O_2$, corresponding to $F_1$ and $F_2$, where $F_1(i,j) = h_ih_j$ for all $i\neq j$, $F_1(i,i) = 0$, $F_2(i,j) = 0$ when $i\neq j$ and $F_2(i,i) = h_i^2$.

\begin{lemma}
\label{lemma: summing rows/columns of outline array}
    If an outline array $O$ exists for an order $k$ frequency array $F$, then for any partition $S_1,S_2,\dots,S_{k'}$ of $[k]$, an outline array $O^*$ exists for the order $k'$ array $F^*$, where for all $i,j\in[k']$ $$F^*(i,j) = \sum_{x\in S_i}\sum_{y\in S_j}F(x,y).$$
\end{lemma}

\begin{lemma}
\label{lemma: frequency array to add on}
    For $3\leq m\leq k$, if $h_m\geq h_{m+1}\geq\dots\geq h_{k}$ and $\sum_{n=m+2}^{k}h_n \leq (m-1)h_1 + (m-2)h_{m+1}$, then an outline array exists corresponding to $F$ of order $k$, where
    $$F(i,j) = \begin{cases}
        0 & \text{if $i=j$}\\
        h_m+h_{m+1} & \text{if $i\neq j$ and $i,j\leq m$}\\
        h_{j} & \text{if $i\leq m$ and $j>m$}\\
        h_{i} & \text{if $i>m$ and $j\leq m$}\\
        0 & \text{otherwise.}
    \end{cases}$$
\end{lemma}
\begin{proof}
    Let $A = \sum_{i=m+1}^k h_i\{i\}$. Since $|A| = \sum_{n=m+1}^{k}h_n \leq (m-1)(h_m+h_{m+1})$, partition $A$ into multisets $A_1,A_2,\dots A_{h_m+h_{m+1}}$, where $0\leq |A_n| \leq m-1$ for all $n\in[h_m+h_{m+1}]$.

    Let $A_n(s)$ be the number of copies of $s$ in $A_n$, for $n\in[h_m+h_{m+1}]$ and $s\in m+[k-m]$. For each $n\in[h_m+h_{m+1}]$, let $F_n$ be a frequency array of order $k$ where
    $$F_n(i,j) = \begin{cases}
        0 & \text{if $i=j$ or $i,j>m$,}\\
        1 & \text{if $i\neq j$ and $i,j\leq m$,}\\
        A_n(i) & \text{if $i>m$ and $j\leq m$,}\\
        A_n(j) & \text{if $i\leq m$ and $j>m$.}
    \end{cases}$$
    An outline array corresponding to $F_n$ is found by taking an $\LS(1^m|A_n|^1)$, which always exists by \Cref{thm: small k squares,thm: a^k square,thm: squaresatmost2}. Use the partition $A_n(m+1),A_n(m+2),\dots,A_n(k)$ of $A_n$ and \Cref{lemma: summing rows/columns of outline array} to obtain the required outline array $O_n$.

    Observe that $F = \sum_{n=1}^{h_m+h_{m+1}} F_n$. Thus, by \Cref{lemma: sum freq arrays}, an outline array exists corresponding to $F$.
\end{proof}

We now use frequency arrays and \Cref{jl_combined_construction} to prove \Cref{conj: largest 3}.

\begin{theorem}
\label{thm: main}
    Let $3\leq m\leq k$. There exists an $\LS(h_m^mh_{m+1}\dots h_k)$ for all $h_m\geq h_{m+1}\geq\dots\geq h_k$.
\end{theorem}
\begin{proof}
    Note that an $\LS(h_k^k)$ exists by \Cref{thm: a^k square}. Now take $m\leq \ell<k$ and suppose that an $\LS(h_{\ell+1}^{\ell+1}h_{\ell+2}\dots h_k)$ exists. If $(\ell-1)(h_{\ell}+h_{\ell+1})<\sum_{i=\ell+1}^kh_i$, then an $\LS(h_\ell^\ell h_{\ell+1}\dots h_k)$ exists by \Cref{jl_combined_construction}. Otherwise, by \Cref{lemma: frequency array to add on}, there exists an outline array $O$ for the frequency array $F$ of order $k$ where
    $$F(i,j) = \begin{cases}
        0 & \text{if $i=j$}\\
        h_\ell+h_{\ell+1} & \text{if $i,j\leq \ell$}\\
        h_{j} & \text{if $i\leq \ell$ and $j>\ell$}\\
        h_{i} & \text{if $i>\ell$ and $j\leq \ell$}\\
        0 & \text{otherwise.}
    \end{cases}$$
    The reduction modulo $(P,P,P)$ for $P=(h_{\ell+1}^{\ell+1}h_{\ell+2}\dots h_k)$ of the $\LS(h_{\ell+1}^{\ell+1}h_{\ell+2}\dots h_k)$ gives an outline array $O'$ of order $k$. Set $O'(i,i) = \emptyset$ for all $i\in[k]$ and then use \Cref{lemma: sum freq arrays} to combine $O'$ with $h_\ell-h_{\ell+1}$ copies of $O$. This gives an outline array corresponding to $F^*$, where
    $$F^*(i,j) = \begin{cases}
        0 & \text{if $i=j$}\\
        h_ih_j & \text{otherwise.}
    \end{cases}$$
    Thus, there exists an $\LS(h_\ell^\ell h_{\ell+1}\dots h_k)$.
\end{proof}

This result proves the existence of many realizations without adding much restriction to the partitions. Although there are limited results for realizations, there is less known about incomplete latin squares.

An \emph{incomplete latin square} of side $n$ and type $h_1\dots h_k$, denoted $\ILS(n;h_1\dots h_k)$, is an order $n$ latin square with pairwise disjoint subsquares of orders $h_1,h_2,\dots, h_k$. A realization is thus an incomplete latin square with $n = \sum_{i=1}^kh_i$.

It is not hard to show that if $n\geq 2\sum_{i=1}^kh_i$, then an $\ILS(n;h_1\dots h_k)$ always exists. It is natural to ask: how close can $n$ get to $\sum_{i=1}^kh_i$ in general? Using our main result, we can significantly decrease this gap.

\begin{theorem}
    If $n\geq 2h_1+\sum_{i=1}^kh_i$, then there exists an $\ILS(n;h_1\dots h_k)$.    
\end{theorem}
\begin{proof}
    Let $n-2h_1-\sum_{i=1}^kh_i = qh_k + r$, where $q\geq 0$ and $0\leq r< h_k$. By \Cref{thm: main}, there exists an $\LS(h_1^3h_2\dots h_k^{q+1}r)$. This is an $\ILS(n;h_1\dots h_k)$.
\end{proof}

The authors believe that the gap $n-\sum_{i=1}^kh_i$ could be further improved from $2h_1$ to $h_1$.

\backmatter

\bmhead{Acknowledgements}
Funding: This work was supported by The Australian Research Council, through the Centre of Excellence for Plant Success in Nature and Agriculture (CE200100015) and the first author would like to acknowledge the support of the Australian Government through a Research Training Program (RTP) Scholarship.

\printbibliography

\end{document}